\documentclass[a4paper,11pt]{article}

\usepackage[dvipsnames]{xcolor}
\usepackage[utf8]{inputenc} 
\usepackage{comment}
\usepackage{enumerate}   
\usepackage{enumitem}
\usepackage{amsmath}
\usepackage{amssymb}
\usepackage{amsthm}
\usepackage{thmtools}
\usepackage{url}
\usepackage{caption}
\usepackage{subcaption}
\usepackage{nicematrix}
\usepackage[top=21 mm, bottom=22 mm, left=22 mm, right= 22 mm]{geometry}
\usepackage[bookmarks=false,hidelinks]{hyperref}

\usepackage{tikz}
\usetikzlibrary{math,calc,intersections}
\usetikzlibrary{positioning,arrows,shapes,decorations.markings,decorations.pathreplacing,matrix,patterns}
\tikzstyle{vertex}=[circle,draw=black,fill=black,inner sep=0,minimum size=5pt,text=white,font=\footnotesize]

\declaretheorem[name=Theorem,numberwithin=section]{theorem}

\newtheorem{lemma}[theorem]{\bf Lemma}

\newtheorem{proposition}[theorem]{\bf Proposition}

\newtheorem*{theorem*}{\bf Theorem}

\theoremstyle{definition}

\def\cF{\mathcal{F}}

\def\cH{\mathcal{H}}

\def\cP{\mathcal{P}}
\def\cQ{\mathcal{Q}}
\def\cS{\mathcal{S}}
\def\cV{\mathcal{V}}
\def\bF{\mathbb{F}}

\def\oF{\overline{\mathbb{F}}}
\def\bE{\mathbb{E}}

\def\bP{\mathbb{P}}
\def\Pb{\mathbb{P}}

\DeclareMathOperator{\aff}{Aff}
\DeclareMathOperator{\Var}{Var}
\DeclareMathOperator{\Cov}{Cov}
\DeclareMathOperator{\diag}{diag}

\title{\vspace{-0.9cm} Point-variety incidences, unit distances and Zarankiewicz's problem for algebraic graphs}
\author{Aleksa Milojevi\'c\thanks{ETH Zurich, e-mail: \textbf{\{aleksa.milojevic, benjamin.sudakov\}@math.ethz.ch}. Research supported in part by SNSF grant 200021\_196965.}, Benny Sudakov\footnotemark[1], Istv\'an Tomon\thanks{Ume\r{a} University, \emph{e-mail}: \textbf{istvan.tomon@umu.se}. Research supported in part by the Swedish Research Council grant VR 2023-03375}}
\date{}

\begin{document}

\maketitle
\begin{abstract}
    In this paper we study the number of incidences between $m$ points and $n$ varieties in $\mathbb{F}^d$, where $\mathbb{F}$ is an arbitrary field, assuming the incidence graph contains no copy of $K_{s,s}$.
    We also consider the analogous problem for algebraically defined graphs and unit distance graphs. 
    
    First, we prove that if $\cP$ is a set of $m$ points and $\cV$ is a set of $n$ varieties in $\bF^{D}$, each of dimension $d$ and degree at most $\Delta$, and in addition the incidence graph is $K_{s,s}$-free, then the number of incidences satisfies \[I(\cP, \cV)\leq O_{d,\Delta, s}(m^{\frac{d}{d+1}} n+m).\]
    This bound is tight when $s,\Delta$ are sufficiently large with respect to $d$, with an appropriate choice of $\bF=\bF(m,n)$. We give two proofs of this upper bound, one based on the framework of the induced Tur\'an problems and the other based on VC-dimension theory. In the second proof, we extend the celebrated result of R\'onyai, Babai and Ganapathy on the number of zero-patterns of polynomials to the context of varieties, which might be of independent interest.

    We also resolve the  problem of finding the maximum number of unit distances which can be spanned by a set of $n$ points $\cP$ in $\bF^d$ whose unit-distance graph is $K_{s, s}$-free, showing that it is
    $$\Theta_{d,s}(n^{2-\frac{1}{\lceil d/2\rceil +1}}).$$

    Finally, we obtain tight bounds on the maximum number of edges of a $K_{s, s}$-free algebraic graph defined over a finite field, thus resolving the Zarankiewicz problem for this class of graphs.
\end{abstract}

\section{Introduction}

In a recent work \cite{MST}, we established sharp bounds on the maximum number of incidences between $m$ points and $n$ hyperplanes in a $d$-dimensional space $\mathbb{F}^d$ for an arbitrary field $\mathbb{F}$, under the standard non-degeneracy condition that the incidence graph contains no copy of the complete bipartite graph $K_{s,s}$. In this paper, we continue this line of research by establishing sharp bounds on the maximum number of edges in incidence graphs between points and varieties, algebraic graphs, and unit distance graphs, under the same non-degeneracy condition.

\subsection{Point-Variety incidences}
Given a set of points $\cP$ and a family of sets $\cV$, the \emph{incidence graph} of $(\cP,\cV)$ is the bipartite graph $G(\cP,\cV)$ with vertex classes $\cP$ and $\cV$ in which an edge is drawn between $x\in \cP$ and $V\in \cV$ if $x\in V$. The number of incidences in $(\cP,\cV)$ is the number of edges of $G(\cP,\cV)$, and it is denoted by $I(\cP,\cV)$. First, we consider the case when the elements of $\cV$ are algebraic varieties, which are defined as common zero sets of a collection of polynomials, see Section~\ref{sec:background} for precise definitions.

The number of incidences between points and varieties in real space have been extensively studied. Pach and Sharir \cite{PS1,PS2} proved that if $\cP$ is a set of $m$ points and $\cV$ is a set of $n$ algebraic curves of degree at most $\Delta$ in $\mathbb{R}^2$ such that $G(\cP,\cV)$ is $K_{s,t}$-free, then $$I(\cP,\cV)=O_{s,t,\Delta}\left(m^{\frac{s}{2s-1}}n^{\frac{2s-2}{2s-1}}+m+n\right).$$
Zahl \cite{Zahl} and Basu and Sombra \cite{BS} extended this bound to dimension 3 and 4, respectively. Subsequently, Fox, Pach, Sheffer, Suk, and Zahl \cite{FPSSZ} established the following general bound in $\mathbb{R}^D$: if $\cP$ is a set of $m$ points and $\cV$ is a set of $n$ varieties in $\mathbb{R}^{D}$ of degree at most $\Delta$, and $G(\cP,\cV)$ is $K_{s,t}$-free, then
\begin{equation}\label{equ:FPSSZ}
    I(\cP,\cV)=O_{D, s,t,\Delta,\varepsilon}\left(m^{\frac{(D-1)s}{Ds-1}+\varepsilon}n^{\frac{D(s-1)}{Ds-1}}+m+n\right),
\end{equation}
for every $\varepsilon>0$. However, this bound is only known to be tight in the special case when $s=2$ and $t$ is sufficiently large, see \cite{Sheffer}.

Here, we consider a similar problem over arbitrary fields. That is, we consider the problem of bounding the number of incidences between points and varieties in $\bF^D$, where $\bF$ is an arbitrary field, under the assumption that the incidence graph is $K_{s, s}$-free (our proofs and results remain essentially the same if we consider the more general problem of forbidding $K_{s,t}$ for $t\geq s$, so we assume $s=t$ to simplify notation).
We obtain the following theorem.

\begin{theorem}\label{thm:varieties upper}
Let $\cP$ be a set of $m$ points and let $\cV$ be a set of $n$ varieties in $\bF^D$, each of dimension $d$ and degree at most $\Delta$. If the incidence graph $G(\cP, \cV)$ is $K_{s, s}$-free, then \[I(\cP, \cV)\leq O_{d,\Delta, s}(m^{\frac{d}{d+1}} n+m).\]
\end{theorem}

Let us compare our bound with (\ref{equ:FPSSZ}). As $s\rightarrow\infty$, the bound in (\ref{equ:FPSSZ}) approaches the shape $m^{\frac{D-1}{D}}n+m+n$, which matches the one in Theorem \ref{thm:varieties upper} in case $d=D-1$. For smaller $d$, Theorem~\ref{thm:varieties upper} provides even better bounds. Curiously, our bound does not depend on the dimension $D$ of the ambient space, it only depends on the degree and the dimension of the varieties. Moreover, our bounds are essentially tight, as the next theorem shows.

\begin{theorem}\label{thm:construction varieties}
Let $d<D$ be positive integers, $\alpha>0$, and $m, n$ be positive integers such that $n=\lfloor m^{\alpha}\rfloor$, and $m,n$ are sufficiently large with respect to $D$ and $\alpha$. Then there exist a prime $p$, a set of $m$ points $\cP\subseteq \bF_p^D$ and a set $\cV$ of $n$ varieties in $\bF_p^D$ of dimension $d$ and degree at most $\Delta=\lceil(1+\alpha)(d+1)\rceil^2$ such that $G(\cP, \cV)$ does not contain $K_{s, s}$ with $s=\Delta^{1/2}$, and $$I(\cP, \cV)\geq \Omega_{d,\alpha}(m^{\frac{d}{d+1}} n).$$
\end{theorem}

The proof of Theorem~\ref{thm:varieties upper} uses the framework of induced Tur\'an problems, as introduced in \cite{HMST, MST}. Namely, we construct a bipartite graph $H_{d,\Delta}$ such that incidence graphs of points and varieties in $\bF^D$ of dimension $d$ and degree at most $\Delta$ avoid $H_{d,\Delta}$ as an induced subgraph, and every vertex of $H_{d,\Delta}$ in one of the parts has degree at most $d+1$. Then, we adapt the ideas of \cite{HMST} to prove bounds on the number of edges of induced $H_{d,\Delta}$-free and $K_{s,s}$-free bipartite graphs. 

Furthermore, we provide an alternative proof of Theorem~\ref{thm:varieties upper} by establishing a bound on the VC-dimension of incidence graphs and the rate of growth of their shatter functions. To do this we bound the number of incidence-patterns between points and lower-dimensional varieties, extending
the celebrated result of R\'onyai, Babai and Ganapathy \cite{RBG} on the number of zero-patterns of polynomials.

We remark that graphs with bounded VC-dimension avoid induced copies of a certain forbidden bipartite subgraph with similar properties as $H_{d,\Delta}$, but we think that our explicit construction of $H_{d,\Delta}$ is more straightforward. Moreover, in certain situations, like the unit-distance problem for instance, finding specific forbidden bipartite graphs leads to better bounds compared to the ones which follow from VC-dimension.

Finally, in order to prove Theorem~\ref{thm:construction varieties} as well as other constructions in this paper, we use the results from \cite{ST23}, which are based on the random algebraic method originally introduced by \cite{Bukh}.

\subsection{Algebraic graphs}
Fox, Pach, Sheffer, Suk, and Zahl \cite{FPSSZ} considered the classical Zarankiewicz's problem for semi-algebraic graphs. This type of problems comes from an old question of Zarankiewicz, who asked what is the largest number of edges that an $n$-vertex graph can have without containing a copy of $K_{s, s}$. In this setting, we say that a bipartite graph $G=(A,B,E)$ is \emph{semi-algebraic} of description complexity $t$ in $(\mathbb{R}^{D_1},\mathbb{R}^{D_2})$, if $A\subset \mathbb{R}^{D_1}$, $B \subset \mathbb{R}^{D_2}$, and the edges are defined by the sign-patterns of $t$ polynomials $f_1,\dots,f_t:\mathbb{R}^{D_1}\times \mathbb{R}^{D_2}\rightarrow \mathbb{R}$, each of degree at most $t$. The authors of \cite{FPSSZ} showed that if such a graph $G$ is $K_{s,s}$-free with $|A|=m$ and $|B|=n$, then
 $$|E(G)|=O_{D_1,D_2,t,s,\varepsilon}\left( m^{\frac{D_2(D_1-1)}{D_1D_2-1}+\varepsilon}n^{\frac{D_1(D_2-1)}{D_1D_2-1}}+m+n\right)$$
for every $\varepsilon>0$. In the special case of interest $D=D_1=D_2$, this gives
\begin{equation}\label{equ:semi-algebraic}
    |E(G)|=O_{D,t,s,\varepsilon}\left((mn)^{\frac{D}{D+1}+\varepsilon}+m+n\right).
\end{equation}

This bound is only known to be tight for $D=2$. Furthermore, it is worth highlighting that this upper bound matches (up to the $\varepsilon$ error term) the best known upper bound on the number of incidences between $m$ points and $n$ hyperplanes in $\mathbb{R}^{D}$, originally proved by Apfelbaum and Sharir~\cite{AS07}.
 
We consider the analogous problem for \emph{algebraic graphs}, that is, graphs that are defined with respect to zero-patterns of polynomials over arbitrary fields. Given two sets of points $\cP\subseteq \bF^{D_1}, \cQ\subseteq \bF^{D_2}$, a collection of $t$ polynomials $f_1, \dots, f_t:\bF^{D_1}\times \bF^{D_2}\to \bF$, and a boolean formula $\Phi$, we define the bipartite graph $G$ on the vertex set $\cP\cup \cQ$, where $x\in \cP$ and $y\in \cQ$ are adjacent if and only if 
\[\Phi([f_1(x, y)=0], \dots, [f_t(x, y)=0])=1.\]
We say that $G$ is an \emph{algebraic graph of description complexity $t$} if it can be described by $t$ polynomials $f_1, \dots, f_t$, each of which has degree at most $t$. 

\begin{theorem}\label{thm:algebraic zarankiewicz}
Let $G$ be an algebraic graph of description complexity at most $t$ on the vertex set $\cP\cup \cQ$, where $\cP\subseteq \bF^{D_1}$, $\cQ\subseteq \bF^{D_2}$ and $|\cP|=m$, $|\cQ|=n$. If $G$ is $K_{s, s}$-free, the number of edges in $G$ is at most $$O_{D_1,D_2,t, s}(\min\{m^{1-1/D_1}n, mn^{1-1/D_2}\}).$$
Moreover, this bound is tight, for every $D_1,D_2$, there exists $t$ such that if $m, n$ are sufficiently large and $s\geq D_1+D_2$, the following holds. There exist a field $\bF$ and a $K_{s, s}$-free algebraic graph $G$ of description complexity at most $t$ in $(\bF^{D_1}, \bF^{D_2})$ with parts of size $m, n$ and at least $\Omega(\min\{m^{1-1/D_1}n, mn^{1-1/D_2}\})$ edges.
\end{theorem}

\subsection{Unit distances in finite fields}

The celebrated Erd\H{o}s unit distance problem \cite{Erdos1,Erdos2} is one of the most notorious open problems in combinatorial geometry, asking to estimate the function $f_d(n)$, the maximum number of unit distances spanned by $n$ points in $\mathbb{R}^d$. The cases $d=2,3$ are the most difficult ones, with still a large a gap between the best known lower and upper bounds. The state-of-the-art is $n^{1+\Omega(1/\log\log n)}<f_2(n)<O(n^{4/3})$, where the lower bound is due to Erd\H{o}s \cite{Erdos2}, and the upper bound is due to Spencer, Szemer\'edi and Trotter \cite{SST}. Furthermore, $\Omega(n^{4/3}\log\log n)<f_3(n)<n^{1.498}$ by \cite{Erdos2} and \cite{Zahl2}, respectively. On the other hand, for $d\geq 4$, it is not difficult to construct a set of points achieving $f_d(n)=\Theta_d(n^2)$ \cite{Lenz}. Indeed, if $d\geq 4$, one can take two orthogonal linear subspaces $V_1$ and $V_2$, each of dimension at least 2. Then, placing $n/2$ points on the origin centered sphere $S_i\subset V_i$ of radius $1/\sqrt{2}$ for $i=1,2$, we get at least $n^2/4$ unit distances.  However, as it is shown in \cite{FPSSZ}, in some sense this construction is the only way one can get a quadratic number of unit distances. More precisely, if we assume that no $s$ points are contained in a $(d-3)$-dimensional sphere, then the maximum number of unit distances is at most 
$$O_{d,s,\varepsilon}(n^{2-\frac{2}{d+1}+\varepsilon})$$
for every $\varepsilon>0$. The key observation is that such unit distance graphs are semi-algebraic of complexity at most $2$ containing no copy of $K_{s,s}$, in which case the previous bound follows simply from (\ref{equ:semi-algebraic}). See also Frankl and Kupavskii \cite{FK} for further strengthening.

Here, we completely resolve the analogous problem of bounding the number of unit distances over arbitrary fields, assuming the unit distance graph is $K_{s,s}$-free. A \textit{unit sphere} in $\bF^d$ with center $a=(a_1, \dots, a_d)\in \bF^d$ is defined as the set of points $(x_1, \dots, x_d)\in \bF^d$ which satisfy $$(x_1-a_1)^2+\dots+(x_d-a_d)^2=1.$$
Furthermore, points $a$ and $b$ are at \emph{unit distance} if $b$ is contained in the unit sphere with center $a$.

Using Theorem \ref{thm:algebraic zarankiewicz} and noting that unit distance graphs are algebraic of complexity $2$ in $(\bF^d,\bF^d)$, we immediately get that the maximum number of unit distances spanned by $n$ points in $\bF^d$ such that the unit distance graph contains no copy of $K_{s,s}$ is $O_{d,s}(n^{2-1/d})$. However, as we show in the next theorem, this can be significantly improved.

\begin{theorem}\label{thm:unit_distance}
Let $\cP$ be a set of $n$ points in $\bF^d$. If the unit distance graph of $\cP$ does not contain $K_{s, s}$, then the number of unit distances spanned by $\cP$ is at most $O_{d,s}(n^{2-\frac{1}{\lceil d/2\rceil +1}})$.
\end{theorem}

Furthermore, this bound is sharp in the following sense.

\begin{theorem}\label{thm:sphere_construction}
Let $n, d$ be positive integers. There exists a constant $s=s(d)$, a finite field $\bF_q$ and a set of $n$ points $\cP\subseteq \bF_q^d$ such that the unit distance graph on $\cP$ does not contain $K_{s, s}$ and  $\cP$ spans $\Omega_d(n^{2-\frac{1}{\lceil d/2\rceil+1}})$ unit distances.
\end{theorem}

\noindent
\textbf{Organization.} For the reader's convenience, we begin by introducing the definitions and basics of algebraic geometry in Section~\ref{sec:background}. Then, in Section \ref{sect:variety}, we present two proofs of Theorem \ref{thm:varieties upper}, one based on extremal graph theory and another slightly weaker variant using shatter functions and VC dimension. We postpone the constructions showing tightness of these bounds to Section~\ref{sect:algebraic}, where we also prove Theorem~\ref{thm:algebraic zarankiewicz}. Finally, in Section \ref{sect:spheres}, we address the unit-distance problem and prove Theorems \ref{thm:unit_distance} and \ref{thm:sphere_construction}.

\section{Preliminaries from algebraic geometry}\label{sec:background}

Much of the following material is based on the Appendix of \cite{DGW} and the excellent book of Cox, Little and O'Shea \cite{CLO}.

Let $\oF$ be the algebraic closure of the field $\bF$. We denote by $\oF[x_1, \dots, x_D]$ the set of polynomials in variables $x_1, \dots, x_D$ with coefficients in $\oF$. An \textit{affine algebraic set} $V$ is a set of common zeros of polynomials $f_1, \dots, f_k\in \oF[x_1, \dots, x_D]$ in $\oF^D$, and this set is denoted by $V(f_1, \dots, f_k)$. Formally, \[V(f_1, \dots, f_k)=\left\{(a_1, \dots, a_d)\in \oF^D: f_i(a_1, \dots, a_d)=0\text{ for all }i\in [k]\right\}.\]

In a similar way, one can define projective algebraic sets. In this case, the ambient projective space $\bP^D(\oF)$ over the field $\oF$ is defined as the set of equivalence classes of points in $\oF^{D+1}\backslash\{(0, \dots, 0)\}$, where $(a_0, \dots, a_{D})\sim (b_0, \dots, b_D)$ if $a_i=\lambda b_i$ for every $i$ for some $\lambda\in \oF$. Note that for every homogeneous polynomial $f\in \oF[x_0, \dots, x_D]$ and every $\lambda\neq 0$, we have $f(a_0, \dots, a_D)=0$ if and only if $f(\lambda a_0, \dots, \lambda a_D)=0$. Hence, the set of zeros of a homogeneous polynomial in $\bP^D(\oF)$ is well-defined. Thus, we can define a \textit{projective algebraic set} determined by homogeneous polynomials $f_1, \dots, f_k\in \oF[x_0, \dots, x_D]$ as the common zero set of these polynomials.

One should think of the projective space as the ``completion" of the affine space, in which one adds certain points at infinity. Hence, to every affine algebraic set one can uniquely associate a projective variety which extends it. Formally, one can define a homogenization $f^h$ of the degree $d$ polynomial $f\in \oF[x_1, \dots, x_D]$ by setting $f^h(x_0, \dots, x_D)=x_0^df(x_1/x_0, \dots, x_D/x_0)$. Then, an affine algebraic set $V(f_1, \dots, f_k)\subseteq \oF^D$ extends to the projective algebraic set defined by polynomials $f_1^h, \dots, f_k^h$. 

Finite unions and arbitrary intersections of algebraic sets are also algebraic. An algebraic set $V$ is \textit{irreducible} if there are no algebraic sets $V_1, V_2\subsetneq V$ for which $V_1\cup V_2=V$, and an irreducible algebraic set is called a \textit{variety}. The decomposition theorem for algebraic sets states that there is a unique way to express an algebraic set as the finite union of varieties, which are then called the \textit{components} of the algebraic set (see Section 6 of Chapter 4 in \cite{CLO}).

Given an affine algebraic set $Y\subset \oF^D$, the \emph{ideal of $Y$}, denoted by $I(Y)$, is defined as the set of polynomials $f\in \oF[x_1, \dots, x_D]$ that vanish on $Y$. Clearly, if $g,h\in \oF[x_1, \dots, x_D]$ are polynomials, then the restrictions of $g$ and $h$ on $Y$ are identical if and only if  $g-h\in I(Y)$. Therefore, the space of polynomial functions on $Y$ is isomorphic to the quotient $\oF[x_1, \dots, x_D]/I(Y)$. The same definition works for projective algebraic sets $Y\subseteq \bP^D(\oF)$, where $I(Y)$ stands for the set of all homogeneous polynomials in $\oF[x_0, \dots, x_D]$ vanishing on $Y$.

Next, we define the dimension and the degree of a variety. We give two definitions of the dimension of a variety (for the proof of their equivalence, see Theorem I.7.5 in \cite{H}). Both of these definitions will be useful in our proofs. According to the first definition, the \textit{dimension} of the variety $V$ is the maximum value of $r$ for which there exists a chain of varieties $V_0, V_1, \dots, V_r$ such that $\emptyset=V_0\subsetneq V_1\subsetneq \dots\subsetneq V_r\subsetneq V$.

Next, we give a definition the dimension based on the ideal. Fix a projective variety $V\subseteq \bP^D(\oF)$ and let $\oF_{t}[x_0, x_1, \dots, x_D]$ be the space of homogeneous polynomials of degree $t$, which is a finite-dimensional vector space over $\oF$. Furthermore, let $I_{t}=I(V)\cap \oF_{t}[x_0, x_1, \dots, x_D]$ be the space of those homogeneous polynomials in $I(V)$ which have degree $t$. Since $I_{t}$ is a subspace of $\oF_{t}[x_0, x_1, \dots, x_D]$, one may consider the dimension of the quotient $\dim \big(\oF_{t}[x_0, x_1, \dots, x_D]/I_{t}\big)$, which is called the \textit{Hilbert function} of $I$. By Hilbert's theorem, the integer function $t\mapsto\dim \big(\oF_{t}[x_0, \dots, x_D]/I_{t}\big)$ is a polynomial for $t$ sufficiently large and it is called the \emph{Hilbert polynomial} of $I(V)$ (see e.g. Proposition 3 on p. 487. in \cite{CLO}). The \emph{dimension} of the variety $V$, denoted by $\dim(V)$, is then the degree of the Hilbert polynomial associated to $I(V)$. Furthermore, the \emph{degree} of the variety, denoted by $\deg V$, is defined as $(\dim V)!$ times the leading coefficient of this polynomials. Finally, the dimension and the degree of the affine variety are simply the dimension and the degree of the associated projective variety.

To get some intuition for these concepts, we remark that if $V=V(f)$ is a variety in $D$-dimensional space defined by a single irreducible polynomial $f$, then the dimension of $V$ is $D-1$ and the degree of $V$ is $\deg f$.

To conclude this section, let us mention two important results that we use in our proofs. The first one concerns the intersections of varieties whose degree and dimension are known and holds for affine and projective varieties alike (for a reference, see Example 8.4.6 in \cite{F}).

\begin{theorem}\label{thm:preliminary}
Let $V_1$ and $V_2$ be algebraic varieties of dimensions $d_1\leq d_2$ in a projective or affine ambient space of dimension $D$ such that $V_1\not\subseteq V_2$. If $Z_1, \dots, Z_k$ are the irreducible components of $V_1\cap V_2$, then $\dim(Z_i)\leq d_1-1$ for all $i$ and $\sum_{i=1}^k \deg(Z_i)\leq \deg V_1\cdot \deg V_2$. 
\end{theorem}

The second statement we use is a uniform version of Hilbert's theorem, which can be used to bound the Hilbert function (for a reference, see e.g. Chapter 9 of \cite{NA}). 

\begin{theorem}\label{thm:uniform_hilbert}
Let $V$ be a projective variety and let $I=I(V)$ be the associated ideal. Then, for every positive integer $t$,  $\dim \big(\oF_{t}[x_0, x_1, \dots, x_D]/I_{t}\big)\leq \deg(V) t^{\dim V}+\dim V.$
\end{theorem}

\section{Point-variety incidences}\label{sect:variety}

\subsection{Upper bounds via extremal graph theory}

In this section, we prove Theorem \ref{thm:varieties upper} through the following approach. We construct a bipartite graph $H_{d, \Delta}$ with the property that $G(\cP, \cV)$ does not contain an induced copy of $H_{d, \Delta}$. Then, we show that any $K_{s, s}$-free bipartite graph with sides of size $m, n$ which does not contain an induced copy of $H_{d, \Delta}$ has at most $O(m^{\frac{d}{d+1}} n)$ edges.

Let us begin by describing the forbidden induced bipartite subgraph $H=H_{d, \Delta}$. Let $A$ and $B$ be the parts of $H_{d,\Delta}$, and $k=2^{\Delta^d}+1$. Part $B$ consists of $d+1$ ``layers" of vertices, which we describe as follows. The first two layers contain one vertex each, denoted by $v_1$ and $v_2$, respectively. For $3\leq \ell\leq d+1$, the $\ell$-th layer has $k^{\ell-2}$ vertices, denoted by $v_\ell^{(i_3, \dots, i_\ell)}$, where $i_3, \dots, i_\ell\in [k]$.

Next, we describe part $A$. For each $3\leq \ell\leq d+1$ and each sequence $(i_3, \dots, i_\ell)\in [k]^{\ell-2}$, we add $k$ vertices $w_{\ell,1}^{(i_3,\dots,i_{\ell})},\dots,w_{\ell,k}^{(i_3,\dots,i_{\ell})}$ to $A$ whose neighbours are $v_1, v_2, v_3^{(i_3)}, v_4^{(i_3, i_4)}, \dots, v_\ell^{(i_3, \dots, i_\ell)}$. Note that every vertex of $A$ has degree at most $d+1$. Furthermore, observe that vertices $v_{t}^{(i_3, \dots, i_t)}$ and $v_r^{(j_3, \dots, j_r)}$ have a common neighbour if and only if one of the sequences $(i_3,\dots, i_t)$ and $(j_3, \dots, j_r)$ is a prefix of the other.

\begin{proposition}\label{prop:forbidden induced subgraph}
If $\cP$ is a set of points and $\cV$ is a set of $d$-dimensional varieties of degree at most $\Delta$, the incidence graph $G(\cP, \cV)$ does not contain $H_{d, \Delta}$ as an induced subgraph, where the vertices of $A$ correspond to points and the vertices of $B$ correspond to varieties.
\end{proposition}
\begin{proof}
Assume that $G(\cP,\cV)$ contains an induced copy of $H_{d,\Delta}$ and let  $V_\ell^{(i_3, \dots, i_\ell)}$ be the variety corresponding to the vertex $v_\ell^{(i_3, \dots, i_\ell)}$. We show by induction on $2\leq \ell\leq d+1$ that one can choose a sequence $i_3,\dots,i_{\ell}$  with the following property. Let $Z_1, \dots, Z_t$ be the irreducible components of the intersection $V_1\cap V_2\cap\dots\cap V_\ell^{(i_3, \dots, i_\ell)}$ that satisfy $\dim(Z_i)\leq d-\ell+1$. Then the union $\bigcup_{i=1}^t Z_i$ contains all points corresponding to the common neighbours of $v_1, \dots, v_\ell^{(i_3, \dots, i_\ell)}$ in $A$.

For $\ell=2$, this statement is obvious, since all components of $V_1\cap V_2$ have dimension at most $d-1$ by Theorem~\ref{thm:preliminary}, noting that $V_1$ and $V_2$ are distinct varieties of dimension $d$. To show the inductive step, suppose we found a sequence $i_3,\dots,i_{\ell}$  satisfying the above property. Let $Z_1, \dots, Z_t$ be the irreducible components of the intersection of these varieties satisfying $\dim(Z_i)\leq d-\ell+1$. By iterated application of Theorem~\ref{thm:preliminary}, one can deduce that $t\leq \Delta^\ell$.

For each variety $V_{\ell+1}^{(i_3, \dots, i_\ell, j)}$, $j\in \{1, \dots, k\}$, let us denote by $T_j$ the set of components among $Z_1, \dots, Z_t$ which are fully contained in $V_{\ell+1}^{(i_3, \dots, i_\ell, j)}$. As $k=2^{\Delta^d}> 2^t$, there exist two sets $T_{j}$ and $T_{j'}$ which are identical. In other words, the varieties $V_{\ell+1}^{(i_3, \dots, i_\ell, j)}$ and $V_{\ell+1}^{(i_3, \dots, i_\ell, j')}$ contain the exact same components $Z_i$. We claim that the sequence $i_3,\dots,i_{\ell},j$ then satisfies the required conditions and suffices to perform the inductive step.

To see this, note that $V_{\ell+1}^{(i_3, \dots, i_\ell, j)}$ intersects all components $Z_i\notin T_j$ in subvarieties of dimension at most $d-\ell$ (or in the empty set). Thus, to complete the inductive step, it suffices to show that all common neighbours of $v_1, \dots, v_{\ell+1}^{(i_3, \dots, i_\ell, j)}$ are contained in the components $Z_i\notin T_j$. If there is a common neighbour $P$ of these vertices in a component $Z_i\in T_j$, it belongs to the variety $V_{\ell+1}^{(i_3, \dots, i_\ell, j')}$, and thus it is connected to the vertex $v_{\ell+1}^{(i_3, \dots, i_\ell, j')}$. But this is impossible, since $v_{\ell+1}^{(i_3, \dots, i_\ell, j)}$ and $v_{\ell+1}^{(i_3, \dots, i_\ell, j')}$ have no common neighbours in $H_{d,\Delta}$.

Taking $\ell=d+1$, we get a sequence $i_3,\dots,i_{d+1}$ such that $V_1,\dots, V_{d+1}^{(i_3,\dots,i_{d+1})}$ has the following property. If $Z_1,\dots,Z_t$ are the $0$-dimensional components of  $V_1\cap\dots\cap V_{d+1}^{(i_3,\dots,i_{d+1})}$, then $\bigcup_{i=1}^{t}Z_i$ contains $k$ points, corresponding to the common neighbours of $v_1,\dots,v_{d+1}^{(i_3,\dots,i_{d+1})}$. However, note that an iterated application of Theorem~\ref{thm:preliminary} implies $t\leq \Delta^{d+1}$ and so $t<k$. But this is a contradiction, since each $Z_i$ is a single point.
\end{proof}

Next, given a bipartite graph $H$, we prove a general upper bound on the number of edges in a $K_{s,s}$-free and induced $H$-free bipartite graph. A  similar statement was proved recently by Hunter and the authors of this paper \cite{HMST}, but their result is not directly applicable in case the host graph is bipartite with parts of unequal sizes. Fortunately, we can reuse most of the key auxiliary lemmas to adapt to this scenario.

Let $H=(A,B,E)$ be a bipartite graph such that every vertex in $A$ has degree at most $k$. Given a graph $G$, we say that a set of vertices $S\subset V(G)$ is $(H,k,s)$-rich if for every $T\subset S$, $|T|\leq k$, we have 
$$|\{v\in V(G)\backslash S:N(v)\cap S=T\}|\geq (4|A|s)^{|A|}.$$
We need the following two results from \cite{HMST}.

\begin{lemma}[Lemma 2.3 in \cite{HMST}]\label{lemma:HMST2.3}
    Let $G$ be a graph not containing $K_{s,s}$. If $G$ contains an $(H,k,s)$-rich independent set $S$ of size $|B|$, then $G$ contains $H$ as an induced subgraph in which $B$ is embedded into $S$.
\end{lemma}

\begin{proposition}[Proposition 2.4 in \cite{HMST}]\label{prop:HMST2.4}
    Let $G$ be a $K_{s,s}$-free graph and $X\subseteq V(G)$ a set of at least $(4|A||B| s)^{4|B|+10}$ vertices in which every $k$-tuple of vertices has at least $(4|A||B| s)^{2 |V(H)|}$ common neighbours. Then $X$ contains an $(H,k,s)$-rich independent set of size $|B|$. 
\end{proposition}

Furthermore, we use the following simple statement about independent sets in hypergraphs.

\begin{proposition}\label{prop:hyp_ind}
    Let $\cH$ be a $k$-uniform hypergraph with $N$ vertices and $M$ edges. Then $\cH$ contains an independent set of size at least $N^{\frac{k}{k-1}}/(4(M+N)^{1/(k-1)})$.
\end{proposition}

\begin{proof}
Let $p=(N/2(M+N))^{1/(k-1)}$. Let $X\subseteq V(\cH)$ be a random sample in which each vertex is included independently with probability $p$. Then $\cH[X]$ contains an independent set of size at least $|X|-e(\cH[X])$, as we can remove a single vertex from every edge to get an independent set. But $\mathbb{E}(|X|-e(\cH[X]))=pN-p^k M\geq \frac{pN}{2}$, so there is a choice for $X$ such that $|X|-e(\cH[X])\geq pN/2\geq N^{\frac{k}{k-1}}/(4(M+N)^{1/(k-1)})$.
\end{proof}

Now we are ready to prove our bound, which is a simple combination of the previous two results and the dependant random choice method \cite{FS}.

\begin{lemma}\label{lemma:induced_free}
Let $G=(U,V;E)$ be a bipartite graph, $|U|=m$ and $|V|=n$, and let $H=(A,B;E)$ be a bipartite graph such that every vertex in $A$ has degree at most $k\geq 2$. If $G$ contains no induced copy of $H$ in which $A$ is embedded to $U$, and $G$ is $K_{s,s}$-free, then
$$E(G)=O_{H,s}\left(m^{\frac{k-1}{k}}n+m\right).$$
\end{lemma}

\begin{proof}
Assume $e(G)\geq K(m^{\frac{k-1}{k}}n+m)$, where $K=K(H,s)$ is specified later. Let $t=(4|A||B| s)^{2 |V(H)|}$ and $z=(4|A||B| s)^{4|B|+10}$ be constants depending only on $H$ and $s$. We show that if $G$ contains no $K_{s,s}$, then $G$ contains an induced copy of $H$ in which $A$ is embedded to $U$. 

Let $u$ be a vertex in $U$, and let $X_0=N_G(u)$. Define the $k$-uniform hypergraph $\mathcal{H}$ on vertex set $X_0$ such that a $k$-element set $D\subset X_0$ is an edge if $D$ has at most $t$ common neighbours in $G$. First, we show that there is a choice for $u$ such that $\cH$ contains an independent set of size at least $z$. By Proposition \ref{prop:hyp_ind}, there is independent set of size at least $|X_0|^{\frac{k}{k-1}}/(4(|X_0|+e(\cH))^{\frac{1}{k-1}})$. Hence, it is enough to show that there exists a choice of  $u\in U$ for which $|X_0|^{k}>(4z)^{k-1} (|X_0|+e(\cH))$.

Choose $u$ uniformly at random from $U$. Then $\bE[|X_0|]=\frac{e(G)}{m}\geq K \frac{n}{m^{1/k}}+K$. On the other hand, for any fixed $k$ element set $D\subset V$ with less than $t$ common neighbours, we have $\mathbb{P}(D\subset X_0)\leq \frac{t}{m}.$ Since there are at most $n^{k}$ such $k$ element sets in $V$, we conclude that $\bE[e(\cH)]\leq \frac{t n^{k}}{m}.$ By Jensen's inequality, we have $\bE[|X_0|^{k}]\geq \bE[|X_0|]^{k}\geq K^{k} \frac{n^{k}}{m}+K^k$, and therefore $$\bE[|X_0|^{k}-(4z)^{k-1}(e(\cH)+|X_0|)]>0$$ when $K$ is sufficiently large. Hence, there exists a choice of $u$ for which $\cH$ contains an independent set of size $z$. Fix such a choice, and let $X\subset X_0$ be such an independent set.

As every $k$ element subset of $X$ has at least $t$ common neighbors, we can apply Proposition \ref{prop:HMST2.4} to find an $(H,k,s)$-rich set $S\subset X$ of size $|B|$. But then, by Lemma~\ref{lemma:HMST2.3}, there is an embedding of $H$ in which $B$ is embedded to $S$, finishing the proof.
\end{proof}

\begin{proof}[Proof of Theorem \ref{thm:varieties upper}]
By Proposition \ref{prop:forbidden induced subgraph}, the incidence graph $G(\cP,\cV)$ does not contain an induced copy of the bipartite graph $H_{d,\Delta}$, where part $A$ is corresponds to points. Also, every vertex of $A$ has degree at most $d+1$, so applying Lemma \ref{lemma:induced_free} with $k=d+1$ and $H=H_{d,\Delta}$ gives the desired bound.
\end{proof}

\subsection{Varieties containment patterns and dual shatter functions}

In this section, we show a variant of Theorem~\ref{thm:varieties upper} using a different approach. Note that we use a slightly stronger assumption that the varieties have bounded description complexity, instead of just bounded degree. However, since the description complexity can be bounded by a function of the degree and the ambient dimension, this assumption does not make a big difference (see e.g. Theorem 2.1.16 in \cite{T}). 

We say that the variety $V$ has \textit{description complexity} at most $t$ if it can be defined using at most $t$ polynomials $f_1, \dots, f_t\in \bF[x_1, \dots, x_D]$ of degree at most $t$.

\begin{theorem}\label{thm:varieties upper variant}
Let $\cP$ be a set of $m$ points in $\bF^D$ and let $\cV$ be a set of $n$ varieties in $\bF^D$, each of dimension $d$ and description complexity at most $t$. If the incidence graph $G(\cP, \cV)$ is $K_{s, s}$-free, then \[I(\cP, \cV)\leq O_{D,t, s}(m^{\frac{d}{d+1}} n+m).\]
\end{theorem}

Let us prepare the proof of this theorem. First, we recall some basic notions from the theory of VC-dimesion. Let $X$ be a ground set and let $\cF$ be a family of subsets of $X$. The \textit{shatter function} of the system $\cF$, denoted by $\pi_\cF(k)$, is defined as the maximum number of distinct intersections of a $k$-element set with members of $\cF$. Formally,
$$\pi_\cF(k)=\max_{A\subset X, |A|=k} |\{A\cap B:B\in \cF\}|.$$
For a sequence of varieties $V_1,\dots,V_k$ and a point $x$ in $\bF^D$, we define the \textit{containment pattern} of $V_1,\dots,V_k$ at $x$ as the set of indices $I\subseteq [k]$ of those varieties $V_i$ which contain $x$, that is $I=\{i\in [k]:x\in V_i\}$. Our first key lemma is an extension of a celebrated result of R\'onyai, Babai and Ganapathy \cite{RBG}  on the number of zero-patterns of polynomials. This statement, which might be of independent interest, implies that the family of possible containment-patterns has a polynomial shatter function.

\begin{lemma}\label{lemma:variety_zero_pattern}
Let $V_1,\dots,V_k$ be a sequence of varieties in $\bF^D$, each of dimension $d$ and description complexity at most $t$. The number of distinct containment patterns of $V_1,\dots,V_k$ is at most $O_{D, t}(k^{d+1})$.
\end{lemma}

The proof of Lemma~\ref{lemma:variety_zero_pattern} combines the ideas from \cite{RBG} with Hilbert's theorem. Hence, before presenting the proof, let us recall the result of R\'onyai, Babai and Ganapathy \cite{RBG}. Using our terminology, given sequence of polynomials $f_1, \dots, f_k\in \bF[x_1, \dots, x_D]$ and point $x\in \bF^D$, the \textit{zero-pattern} of $f_1,\dots,f_k$ at $x$ is the set of indices $I\subset [k]$ defined as $I=\{i\in [k]:f_i(x)=0\}$. If $\mathcal{Z}(f_1, \dots, f_k)$ is the set of all zero-patterns of $f_1, \dots, f_k$ and $f_1, \dots, f_k$ are polynomials of degree at most $\Delta$, then the number of distinct zero-patterns is at most 
\begin{equation}\label{eqn:zero-patterns polynomials}
    |\mathcal{Z}(f_1,\dots, f_k)|\leq  \binom{k\Delta}{D}.
\end{equation}
Note that (\ref{eqn:zero-patterns polynomials}) can be used as a black box to derive a weaker version of Lemma~\ref{lemma:variety_zero_pattern}, since every containment patterns of $V_1, \dots, V_k$ corresponds to a zero-patterns of their defining polynomials. However, this argument would lead to the weaker upper bound $O_t(k^D)$ in Lemma \ref{lemma:variety_zero_pattern}.

\begin{proof}[Proof of Lemma~\ref{lemma:variety_zero_pattern}.]
We work in the algebraic closure $\oF$ of the field $\bF$, which can only increase the number of containment patterns. We show that for every $a\in [k]$, the number of containment patterns $I\subset [k]$ with $a\in I$, i.e. the patterns coming from points $x\in V_a$, is at most $O_{t,D}(k^d)$. Then we are clearly done. Without loss of generality, let $a=1$, and let $M$ be the number of containment patterns $I\subseteq [k]$ with $1\in I$.

For $i=2,\dots,k$, let the defining polynomials of the variety $V_i$ be $f_{i,1}, \dots, f_{i,t}\in \oF[x_1, \dots, x_D]$, where we may repeat some of the polynomials to ensure that we have exactly $t$ polynomials. Since every containment pattern corresponds to a unique zero-pattern of polynomials $f_{2,1}, \dots, f_{k,t}$, we have $M\leq N$, where $N$ is the number of zero-patterns of $f_{2,1}, \dots, f_{k,t}$ on the variety $V_1$. Let $\widetilde{f}_{i,j}$ be the polynomial function induced on $V_1$ by the polynomial $f_{i, j}$, i.e. $\widetilde{f}_{i,j}=f_{i,j}\mod I(V_1)$, where we recall that $I(V_1)$ is the ideal of the variety $V_1$ (see the Preliminaries).

Let $x_1, \dots, x_N\in V_1$ be points witnessing the $N$ distinct zero-patterns, and for $j\in [N]$, let $S_j\subset \{\widetilde{f}_{2,1},\dots,\widetilde{f}_{k,t}\}$ be the set of polynomial $\widetilde{f}$ for which $\widetilde{f}(x_j)\neq 0$.  Define the polynomial $$g_j(x)=\prod_{\widetilde{f}\in S_j} \widetilde{f}(x).$$ 
Clearly, we have $g_{j}(x_j)\neq 0$, and $$\deg g_j\leq \sum_{i,\ell} \deg \widetilde{f}_{i,\ell} \leq kt^2.$$ Furthermore, note that $g_{j}(x_{\ell})=0$ if $S_{j}\not\subset S_{\ell}$, since for every $\widetilde{f}\in S_{j}\setminus S_{\ell}$, we have $\widetilde{f}(x_{\ell})=0$.

We now argue that the polynomials $g_1,\dots,g_N$ are linearly independent. Otherwise, there exist scalars $\lambda_1, \dots, \lambda_N\in \oF$, not all zero, for which  \[\sum_{j=1}^n \lambda_j g_j=0.\] Choose $\ell$ such that $\lambda_{\ell}\neq 0$ and $|S_{\ell}|$ is minimal. If we plug  $x=x_{\ell}$ into the above equation, all terms except $g_{\ell}(x_{\ell})$ vanish. Indeed, for all $j\neq \ell$ we either have $\lambda_j=0$ or $S_j\not\subset S_{\ell}$, meaning that $g_j(x_{\ell})=0$. But $g_{\ell}(x_{\ell})\neq 0$, which is a contradiction. Thus, the polynomials $g_1, \dots, g_N$ are linearly independent. But the dimension of the space of polynomials of degree at most $kt^2$ on the variety $V_1$ is bounded by $O_{t, D}(k^{d})$ by Theorem~\ref{thm:uniform_hilbert}, and so $M\leq N\leq O_{t, D}(k^d)$. 
\end{proof}

In the proof of Theorem \ref{thm:varieties upper variant}, we also use the following result of \cite{FPSSZ}, whose proof follows from the Haussler packing lemma \cite{Haussler}. 

\begin{theorem}[Theorem 2.1 from \cite{FPSSZ}]\label{thm:zarankiewicz bounded VC}
Let $c>0$ and $r,s\in \mathbb{N}$, then there exists $c_1=c_1(c,D,s)>0$ such that the following holds. Let $G=(A, B, E)$ be a bipartite graph with $|A|=m$ and $|B|=n$ such that the set system $\cF_1=\{N(q):q\in A\}$  satisfies $\pi_{\cF_1}(k)\leq ck^r$ for every positive integer $k$. Then, if $G$ is $K_{s, s}$-free, we have 
\[|E(G)|\leq c_1(nm^{1-1/r} +m).\]
\end{theorem}

\begin{proof}[Proof of Theorem \ref{thm:varieties upper variant}]
Let $\cV=\{V_1,\dots,V_n\}$, and let $\cF$ be the set of containment patterns of $V_1,\dots,V_n$. Then $\pi_{\cF}(k)=O_{D,t}(k^{d+1})$ by Lemma \ref{lemma:variety_zero_pattern}. For every $x\in \cP$, the neighborhood $N(x)$ corresponds to the containment pattern $I_x=\{i\in [n]:x\in V_i\}$. Therefore, if $\cF_1=\{N(x):x\in \cP\}$, then $\pi_{\cF_1}(k)\leq \pi_{\cF}(k)=O_{D,t}(k^{d+1})$. Applying Theorem \ref{thm:zarankiewicz bounded VC} with $r=d+1$, we get $$I(\cP,\cV)=O_{D,t,s}(n m^{\frac{d}{d+1}}+m).$$
\end{proof}

\section{Algebraic Zarankiewicz's problem}\label{sect:algebraic}

In this section, we present the proof of our bound on Zarankiewicz's problem in algebraic graphs of bounded complexity. We will prove the upper and lower bounds separately, and we will then use the lower bound to show Theorem~\ref{thm:construction varieties}. In the proof of the upper bound, we combine Theorem~\ref{thm:zarankiewicz bounded VC} with the estimate~(\ref{eqn:zero-patterns polynomials}). 

For the lower bounds, we use the random polynomial method, which was pioneered by Bukh \cite{Bukh}. The main idea is to choose the polynomial defining the algebraic graph to be a random low-degree polynomial, chosen uniformly among all polynomials of degree at most $\Delta$ in $\bF[x_1, \dots, x_D]$, where $\bF$ is a finite field of appropriate size. The key observation is that for any fixed $x_1, \dots, x_s\in \bF^D$ and random polynomial $f$ of degree at most $\Delta$, the random variables $\{f(x_j)| j\in [s]\}$ are independent and uniform in $\bF$ as long as $s\leq \min\{\Delta, |\bF|^{1/2}\}$ (see Claims 3.3 and 3.4 in \cite{ST23}).

\begin{theorem}\label{thm:algebraic zarankiewicz upper}
Let $G$ be an algebraic graph of description complexity at most $t$ on the vertex set $\cP\cup \cQ$, where $\cP\subseteq \bF^{D_1}$, $\cQ\subseteq \bF^{D_2}$ and $|\cP|=m$, $|\cQ|=n$. If $G$ is $K_{s, s}$-free, the number of edges in $G$ is at most $$O_{D_1,D_2,t, s}(\min\{m^{1-1/D_1}n, mn^{1-1/D_2}\}).$$
\end{theorem}
\begin{proof}
Let the graph $G$ be defined by the polynomials $f_1, \dots, f_t:\bF^{D_1}\times \bF^{D_2}\to \bF$ and a boolean formula $\Phi$, where $\deg f_i\leq t$ for $i=1, \dots, t$.

Since the roles of $m$ and $n$ are symmetric, it suffices to show that any $K_{s, s}$-free algebraic graph $G$ in $(\bF^{D_1}, \bF^{D_2})$ has at most $O_{t, s}(mn^{1-1/D_2}+n)$ edges. To do this, we use Theorem~\ref{thm:zarankiewicz bounded VC}. Define the set system $\cF=\{N(v):v\in \cQ\}$ on the ground set $\cP$. Our main goal is to show $\pi_{\cF}(k)\leq O_{t, D_2}(k^{D_2})$ for all $k$, since in this case Theorem~\ref{thm:zarankiewicz bounded VC} shows that $e(G)\leq O_{t, s}(mn^{1-1/D_2}+n)$. To this end, we fix  $k$ points $x_1, \dots, x_k\in \cP$ and bound the number of possible intersections $\{x_1, \dots, x_k\}\cap N(v)$ for $v\in \cQ$.  For every $i\in [k]$ and $j\in [t]$, consider the polynomial $y\mapsto f_j(x_i, y)$ over $\bF^{D_2}$ Let $F$ be the set of these $kt$ polynomials.

The number of distinct intersections $\{x_1, \dots, x_k\}\cap N(v)$ is bounded by the number of zero-patterns of polynomials in $F$, since the zero-pattern of polynomials in $F$ at a point $v\in \bF$ can be used to determine which $x_i$ is adjacent to $v$. Since $F$ is a set of $kt$ polynomials of degree at most $t$, by (\ref{eqn:zero-patterns polynomials}) we have that the number of zero-patterns of the polynomials in $F$ is bounded by $\binom{kt^2}{D_2}$. Therefore, $\pi_{\cF}(k)\leq \binom{kt^2}{D_2}= O_{t, D_2}(k^{D_2})$, finishing the proof.
\end{proof}

To show the lower bounds, we will need the following simple lemma.

\begin{lemma}\label{lemma:zeros of random polynomials}
Let $\Delta, D\geq 3$ be fixed integers and let $p\geq 4$ be a prime. Let $f\in \bF_p[x_1, \dots, x_D]$ be chosen uniformly at random among polynomials of degree at most $\Delta$, then $f$ has at least $p^{D-1}/2$ zeros in $\bF_p^D$ with probability at least $\frac{3}{4}$.
\end{lemma}
\begin{proof}
Since $f(x)$ is a uniformly distributed in $\bF_p$ for any fixed $x\in \bF_p^D$, the expected number of zeros of $f$ is $\bE\big[|V(f)|\big]=p^{D-1}$. On the other hand, using the second moment method, one can also show that the random variable $|V(f)|$ is concentrated around its mean. More precisely, if we denote by $\mathbf{1}_x$ the indicator random variable of $f(x)=0$, we have 
\[\Var\big(|V(f)|\big)= \sum_{x, y\in \bF^D} \Cov\big(\mathbf{1}_x, \mathbf{1}_y\big).\]
Claim 3.4 of \cite{ST23} states that $f(x)$ and $f(y)$ are independent random variables when $x\neq y$, if $p\geq 4$ and $\Delta\geq 2$. Moreover, for $x=y$, we have $\Cov\big(\mathbf{1}_x, \mathbf{1}_x\big)\leq \Pb[f(x)=0]=p^{-1}$. Hence, $\Var\big( |V(f)|\big)\leq p^{D-1}$. By Chebyshev's inequality, we have $\Pb[ |V(f)|<p^{D-1}-\lambda \sqrt{p^{D-1}}]\leq \frac{1}{\lambda^2}$, which for $\lambda=\frac{1}{2}\sqrt{p^{D-1}}$ gives 
\[\Pb\left[|V(f)|<\frac{p^{D-1}}{2}\right]\leq \frac{4}{p^{D-1}}\leq \frac{1}{4}.\]
\end{proof}

\begin{proposition}\label{prop:algebraic zarankiewicz lower}
For any $D_1, D_2$ and sufficiently large integers $m, n$, there exists a field $\bF$, sets of points $\cP\subseteq \bF^{D_1}, \cQ\subseteq \bF^{D_2}$, $|\cP|=m, |\cQ|=n$ and a polynomial $f:\bF^{D_1}\times \bF^{D_2}\to \bF$ of degree at most $\Delta=(D_1+D_2)^2$ with the following property. The algebraic graph $G$ defined on $\cP\cup \cQ$ by the equation $f(x, y)=0$ is $K_{s, s}$-free, where $s=D_1+D_2$, and has $\Omega(\min\{m^{1-1/D_1}n, mn^{1-1/D_2}\})$ edges.
\end{proposition}
Observe that the graph $G$ has description complexity at most $(D_1+D_2)^2$ and therefore Proposition~\ref{prop:algebraic zarankiewicz lower} suffices to show the second part of Theorem~\ref{thm:algebraic zarankiewicz}.

\begin{proof}
By symmetry, we may assume that $mn^{1-1/D_2}<m^{1-1/D_1}n$, which is equivalent to $m\leq n^{D_1/D_2}$. Let $p$ be the smallest prime larger than $n^{1/D_2}$, which satisfies $p<2n^{1/D_2}$ by Bertrand's postulate.

We define an algebraic graph $G_0$ in $(\bF_p^{D_1}, \bF_p^{D_2})$  as follows. Let $\cP_0=\bF_p^{D_1}$ and $\cQ_0=\bF_p^{D_2}$. Furthermore, let $f:\bF_p^{D_1}\times \bF_p^{D_2}\to \bF_p$ be a polynomial chosen randomly from the uniform distribution on all polynomials  of degree at most $\Delta=(D_1+D_2)^2$. Finally, set $x\in \cP_0$ and $y\in \cQ_0$ to be adjacent in $G_0$ if and only if $f(x, y)=0$.

We begin by showing that the probability $G_0$ contains $K_{s, s}$ is at most $\frac{1}{4}$. Given $2s$ points $x^{(1)}, \dots, x^{(s)}\in \cP_0, y^{(1)}, \dots, y^{(s)}\in \cQ_0$, the $s^2$ points $(x^{(i)}, y^{(j)})$ are independent uniform random variables as long as $s^2\leq \min\{\Delta, p^{1/2}\}$, see \cite{Bukh}. Therefore, the probability  that $f(x^{(i)}, y^{(j)})=0$ for all $i, j\in [s]$ is exactly $p^{-s^2}$.  Hence, we can apply the union bound over all subsets of $\cP_0$ and $\cQ_0$ of size $s$ to deduce
\[\Pb[K_{s, s}\subseteq G]\leq \binom{p^{D_1}}{s}\binom{p^{D_2}}{s} p^{-s^2}\leq \frac{p^{D_1s}}{s!}\frac{p^{D_2s}}{s!}p^{-s^2}\leq \frac{1}{(s!)^2}\leq \frac{1}{4}.\]

On the other hand, Lemma~\ref{lemma:zeros of random polynomials} shows that with probability at least $3/4$, the polynomial $f$ has at least ${p^{D_1+D_2-1}}/{2}$ zeros, which means that $G_0$ has at least  ${p^{D_1+D_2-1}}/{2}$ edges with probability at least $3/4$. Therefore, there exists a polynomial $f$ such that the graph $G_0$ is $K_{s, s}$-free and has at least $\frac{1}{2p}|\cP_0||\cQ_0|$ edges. Let us fix this graph $G_0$.

Then we get our final graph $G$ by sampling an $m$ element subset $\cP$ of $\cP_0$, and an $n$ element subset $\cQ$ of $\cQ_0$. Then, the expected number of edges of $G$ is $\frac{|\cP|}{|\cP_0|}\frac{|\cQ|}{|\cQ_0|}e(G_0)\geq \frac{mn}{2p}$. Hence, there exists a choice of $\cP, \cQ$ for which $e(G)\geq \frac{mn}{2p}$, let us fix this choice. This graph $G$ is clearly algebraic and $K_{s, s}$-free, since it is an induced subgraph of a $K_{s, s}$-free algebraic graph. Furthermore,  $e(G)\geq \frac{mn}{2p}\geq \frac{mn}{2} (2n)^{-1/D_2}=\Omega(mn^{1-1/D_2})$. This completes the proof. 
\end{proof}

We conclude the section by proving Theorem~\ref{thm:construction varieties}, which demonstrates that our upper bounds on point-variety incidences are tight.

\begin{proof}[Proof of Theorem \ref{thm:construction varieties}]
Consider the case $d=D-1$. The general case follows by embedding $\bF_p^{d+1}$ into a higher dimensional space. Let $p$ be a prime between $m^{1/D}$ and $2m^{1/D}$, which exists by Bertrand's postulate. Further, let $D'=\lceil \alpha D\rceil$, and let $G$ be an algebraic graph in $(\bF_p^D, \bF_p^{D'})$ with parts $\cP, \cQ$ of size $m, n$ obtained through Proposition~\ref{prop:algebraic zarankiewicz lower}. 

We construct a set of hypersurfaces $\cV$ from the set of points $\cQ$. For each point $q\in \cQ$, we consider the polynomial $f_q:\bF_p^{D}\to \bF_p$ given by $f_q(x)=f(x, q)$. Since $f_q(x)$ may not be irreducible, let $g_q(x)$ be an irreducible factor of $f_q(x)$ with the largest number of roots in $\cP$. Finally, define the algebraic set $V_q=\{x\in \bF_p^D|g_q(x)=0\}$. We show that the point set $\cP$ and the collection $\cV=\{V_q|q\in \cQ\}$ satisfy all conditions of the theorem.

Since $g_q(x)$ is an irreducible polynomial, the ideal $\langle g_q(x)\rangle \subseteq \bF_p[x_1, \dots, x_D]$ is a prime ideal, and therefore by Proposition 3, page 207 of \cite{CLO}, the algebraic set defined by $g_q(x)$ is irreducible and therefore a variety. Furthermore, the dimension of the variety $V_q=\{x\in \bF_p^D| g_q(x)=0\}$ is $D-1=d$. Hence, $\cV$ is indeed a collection of varieties of required dimension. 

Let us now argue that there are no disctinct varieties $V_{q_1}, \dots, V_{q_s}\in \cV$ with at least $s$ points of $\cP$ in common, where $s=(D+D')^2$. If such $s$ varieties exist, the common neighbourhood of points $q_1, \dots, q_s$ in the graph $G$ contains at least $s$ points, which is not possible since $G$ is $K_{s, s}$-free. Hence, the incidence graph $G(\cP, \cV)$ is $K_{s, s}$-free.

Finally, we argue that there are many incidences between $\cP$ and $\cV$, i.e. that $I(\cP, \cV)\geq \Omega_{D, \alpha}(m^{\frac{D-1}{D}}n)$. The main observation is that the polynomials $f_q(x)$ have at most $(D+D')^2$ irreducible factors, since $\deg f_q(x)\leq (D+D')^2$. Hence, the number of roots of $g_q(x)$ in $\cP$ is at least $\frac{1}{(D+D')^2}$ times the number of roots of $f_q(x)$ in $\cP$. But the number of roots of $f_q(x)$ in $\cP$ summed over all $q\in \cQ$ is the number of edges of $G$ and therefore we have:
\begin{align*}
    I(\cP, \cV)=\sum_{q\in \cQ} |V_q\cap \cP|\geq \sum_{q\in \cQ} \frac{|\{f_q(x)=0\}\cap \cP|}{(D+D')^2}\geq \frac{e(G)}{(D+\lceil \alpha D\rceil )^2}\geq \Omega_{D, \alpha}\Big(\min\{m^{\frac{D-1}{D}} n, mn^{\frac{D'-1}{D'}}\}\Big).
\end{align*}
It is not hard to see that $m^{-1/D}\leq n^{-1/D'}$ since $m^{D'}=m^{\lceil \alpha D\rceil }\geq m^{\alpha D}=n^D$. Therefore, $I(\cP, \cV)\geq \Omega_{D, \alpha}\Big(m^{\frac{D-1}{D}} n\Big)$, which suffices to complete the proof.
\end{proof}

\section{Unit distances}\label{sect:spheres}

\subsection{Geometry of spheres}

In this section, we discuss basic properties of spheres in finite fields, highlighting some differences with the real space. Throughout this section, we consider a non-degenerate bilinear form $\langle \cdot, \cdot \rangle:\bF^d\to \bF$, with respect to which we define the notions of orthogonality and distance. The form is non-degenerate if for every $v\neq 0$ there exists some $w$ such that  $\langle v, w \rangle\neq 0$. It is standard to associate a norm to a bilinear form by defining $\|v\|^2=\langle v, v\rangle$. In what follows, a \emph{unit sphere} with center $w$ is defined to be the set of points $x\in \bF^d$ for which $\|x-w\|^2=1$.

Furthermore, we say that vectors $u, v\in \bF^d$ are \textit{orthogonal} if $\langle u, v\rangle=0$. One can generalize the notion to affine flats and say that $U, V\subseteq \bF^d$ are \textit{orthogonal} if for all points $u, u'\in U$ and $v, v'\in V$ one has $\langle u-u', v-v'\rangle =0$. A key property of orthogonal subspaces we use is that they satisfy $\dim U+\dim V\leq d$. Furthermore, for vectors $v_1, \dots, v_k\in \bF^d$, we denote the \textit{affine span} of $v_1, \dots, v_k$ by $\aff\{v_1, \dots, v_k\}$. More formally, we have $\aff\{v_1, \dots, v_k\}=\{\sum_{i=1}^d \lambda_i v_i | \lambda_i\in \bF_q, \lambda_1+\dots+\lambda_d=1\}$.

Finally, we define a specific bilinear form $\langle u, v\rangle_d$ and the associated norm $\|v\|_d$ on the space $\bF_q^d$ as follows. When $d\not\equiv 1\bmod 4$, we let $\langle u, v\rangle_d$ be the standard inner product defined by $\langle u, v\rangle_d=u_1v_1+\dots+u_dv_d$, while for $d\equiv 1\bmod 4$ we define $\langle u, v\rangle_d=u_1v_1+\dots+u_{d-1}v_{d-1}-u_dv_d$. In both cases, we define the norm $\|v\|_d^2=\langle v, v \rangle_d$.

\begin{lemma}\label{lemma:intersections of spheres}
Let $\langle \cdot, \cdot\rangle$ be a nondegenerate bilinear form on $\bF^d$ and $S_1, \dots, S_k$ be unit spheres in $\bF^d$, defined with respect to $\langle\cdot, \cdot\rangle$. Then there exists an affine subspace $U\subseteq \bF^d$ such that $S_1\cap \dots\cap S_k=S_1\cap U$. Moreover, $U$ is orthogonal to the affine flat spanned by the centers of $S_1, \dots, S_k$.
\end{lemma}
\begin{proof}
We prove the statement by induction on $k$. Let us denote the centers of the spheres $S_1, \dots, S_k$ by $w_1, \dots, w_k$.

When $k=1$, we may take $U=\bF^d$. For $k>1$, we assume by our induction hypothesis that $S_1\cap \dots\cap S_{k-1}=S_1\cap U$ for some affine subspace $U\subseteq \bF^d$, which is orthogonal to $\aff\{w_1, \dots, w_{k-1}\}$. Hence, we have $S_1\cap \dots\cap S_{k-1}\cap S_k=S_1\cap S_k\cap U$. Consider the defining equations for $S_1$ and $S_k$, which are \[\langle x-w_1, x-w_1\rangle=1 \text{ and } \langle x-w_k, x-w_k\rangle=1.\] By subtracting these two equations, we have $2\langle x, w_k-w_1\rangle+\langle w_1, w_1\rangle-\langle w_k, w_k\rangle=0$ for any $x\in S_1\cap S_k$. Note that this equation is linear in $x$ and hence it defines a hyperplane $H$, whose normal vector is $w_1-w_k$. We conclude that $S_1\cap S_k=S_1\cap H$ and therefore $S_1\cap \dots \cap S_k=S_1\cap S_k\cap U=S_1\cap (U\cap H)$, where $U\cap H$ is an affine space. 

Let us now argue that $U\cap H$ is orthogonal to $\aff\{w_1, \dots, w_k\}$. Let $u, u'\in U\cap H$ and $v, v'\in \aff\{w_1, \dots, w_k\}$ be arbitrary vectors. One may write $v-v'=\sum_{i=2}^k \lambda_i (w_i-w_1)$ with suitable $\lambda_2,\dots,\lambda_k\in \bF$. Therefore, we have 
\begin{align*}
\langle u-u', v-v'\rangle =\Big\langle u-u', \sum_{i=2}^{k-1} \lambda_i (w_i-w_1)\Big\rangle+\left\langle u-u', \lambda_k (w_k-w_1)\right\rangle.
\end{align*}
Note that the first term is zero since $u, u'\in U$, while the second term is zero because $u, u'\in H$. Hence, $U\cap H$ and $\aff\{w_1 \dots, w_k\}$ are orthogonal. 
\end{proof}

One counter-intuitive feature of spheres over finite fields is that they can contain flats. In the next lemma, we study some properties of such flats. 

\begin{lemma}\label{lemma:isotropic flats}
Let $\bF$ be a field of characteristic different from 2, let $V$ be a flat, and let $S$ be a unit sphere in $\bF^d$ centered at $w$. If $V\subseteq S$, then for any $x, y\in V$ we have $\langle x-y, x-y\rangle =0$ and $\langle x-w, x-y\rangle =0$.
\end{lemma}
\begin{proof}
If $x, y\in V$, then $x+\lambda (y-x)\in V$ for every $\lambda\in \bF$. Since $S=\{v\in \bF^d:\langle v-w, v-w\rangle=1\}$ and $V\subseteq S$, we must have 
\[1=\langle x+\lambda (y-x)-w, x+\lambda (y-x)-w\rangle=\langle x-w, x-w\rangle + 2\lambda \langle x-w, y-x\rangle +\lambda^2 \langle y-x, y-x\rangle.\]
Since $\langle x-w, x-w\rangle =1$, we obtain $\lambda^2 \langle y-x, y-x\rangle+2\lambda \langle x-w, y-x\rangle =0$ for all $\lambda$, implying that $\langle y-x, y-x\rangle=0$ and $\langle x-w, y-x\rangle=0$, just as claimed. 
\end{proof}

A flat $V$ for which $\langle x-y, x-y\rangle =0$ for all $x, y\in V$ is called \textit{totally isotropic}.

\begin{lemma}\label{lemma:auxiliary odd d}
Let $\bF$ be a field containing no roots of the equation $x^2=-1$ and let $d=2k+1$. Furthermore, consider the bilinear form $\langle\cdot, \cdot\rangle_d$ on $\bF^d$. Then $\bF^d$ does not contain a pair $(V,w)$, where $V$ is a totally isotropic flat of dimension $k$, and $w$ is a vector of norm $1$ orthogonal to $V$.
\end{lemma}
\begin{proof}
Assume that $\bF^d$ contains such a pair $(V,w)$. After translation, we may assume that $V$ contains the origin. Let $v_1, \dots, v_k$ be a basis of the subspace $V$ which satisfies $\langle v_i, e_j\rangle_d =\delta_{ij}$ for all $i, j\in [k]$, where $e_1, \dots, e_d$ are the standard basis vectors. Here, $\delta_{ij}$ is defined to be $1$ when $i=j$ and $0$ otherwise. After possibly rearranging the coordinates, Gaussian elimination shows that it always possible to find such a basis. Furthermore, since $V$ is totally isotropic, any vector $w'=w+v$ with $v\in V$ is also of unit norm and orthogonal to $V$. Therefore, we may subtract an appropriate linear combination of $v_1, \dots, v_k$ from $w$ and assume that $\langle w, w\rangle_d =1$ and $\langle w, e_j\rangle_d=0$ for all $j\in [k]$.

To derive a contradiction, we compute the Gram matrix $G$ of the vectors $v_1'=v_1-e_1, \dots, v_k'=v_k-e_k, w$. Note that the first $k$ coordinates of each of these $k+1$ vectors are equal to $0$ and therefore one might consider $v_1', \dots, v_k', w$ as vectors in $\bF^{k+1}$ without changing their Gram matrix. Hence, we have $G(i,j)=\langle v_i', v_j'\rangle_d =\langle v_i, v_j\rangle_d -\langle e_i, v_j\rangle_d -\langle v_i, e_j\rangle_d +\langle e_i, e_j\rangle_d=-\delta_{ij}$ for all $i,j\in [k]$. Similarly, one can compute $G(i,k+1)=G(k+1,i)=\langle v_i', w\rangle_d=\langle v_i, w\rangle_d -\langle e_i, w\rangle_d =0$ and $G(k+1,k+1)=\langle w, w\rangle_d=1$. In conclusion, $G=\diag[-1,\dots,-1,1]$. 

On the other hand, let $A$ be a $(k+1)\times (k+1)$ the matrix whose columns are the vectors $v_1', \dots, v_k', w$ (recall that we may consider these vectors as elements of $\bF^{k+1}$). Recalling the definition of the the bilinear form $\langle v_i', v_j'\rangle_d$, we note that we can rewrite it as $\langle v_i', v_j'\rangle_d=(v_i')^T M v_j'$, where $M=I_{k+1}$ if $d\equiv 3\bmod 4$ and $M=\diag[1, \dots, 1, -1]$ if $d\equiv 1\bmod 4$. Hence, one can express the Gram matrix as $G=A^T M A$ and so $A^TMA=\diag[-1, \dots, -1, 1]$. By taking determinants, we see that $\det(A)^2\det(M)=\det(\diag[-1, \dots, -1, 1])=(-1)^k$. 

The final observation is that $\det(M)=-1$ when $d\equiv 1\bmod 4$ and $\det(M)=1$ otherwise, meaning that $\det(M)=(-1)^{k-1}$. Hence, no matter the parity of $k$, we have $\det(A)^2=-1$. But this is impossible, since the equation $x^2=-1$ has no solution in $\bF$. 
\end{proof}

\subsection{Upper bounds}

In this section, we prove Theorem \ref{thm:unit_distance}. In particular, we prove the following incidence bound between points and unit spheres defined with respect to any nondegenerate bilinear form, from which the theorem immediately follows.

\begin{proposition}\label{prop:point-sphere incidences}
Let $\cP$ be a set of $n$ points  and let $\cS$ be a collection of $n$ unit spheres in $\bF^d$. If the incidence graph $G(\cP, \cS)$ is $K_{s, s}$-free, then the number of incidences between $\cP$ and $\cS$ is at most $O_{d,s}(n^{2-\frac{1}{\lceil d/2\rceil +1}})$.
\end{proposition}

Following the ideas of \cite{MST}, we prove that  there exists a simple family $\mathcal{F}_d$ of graphs with the property that an incidence graph of points and spheres in $\bF^{d-1}$ does not contain any member of $\mathcal{F}_d$ as an induced subgraph. The family $\mathcal{F}_d$ is defined to contain all bipartite graphs on vertex classes $\{a_1,\dots,a_d\}$ and $\{b_1,\dots,b_d\}$, where $a_ib_j$ is an edge for $i\geq j-1$, and $a_ib_{i+2}$ is a non-edge for $i=1,\dots,d-2$. 

A simpler way to visualise this family is through the notion of a \textit{pattern}. A \textit{pattern} is an edge labeling of a complete bipartite graph, where every edge is labeled $0$, $1$ or $*$. A pattern $\Pi$ corresponds to a family of graphs $\cF_\Pi$ on the same vertex set, which contains the graphs $F$ such that all edges of the pattern labeled by $0$ do not appear in $F$, all edges labeled by $1$ appear in $F$, while the edges labeled by $*$ may or may not appear. Finally, we say that a graph $G$ contains a pattern $\Pi$ if it contains an induced copy of some graph in the family $\cF_\Pi$.

One can express the condition that the incidence graph does not contain an induced member of the family $\cF_d$ by forbidding the following pattern. Namely, let $\Pi_{d}$ be $2d$ vertices $a_1, \dots, a_{d}$, $b_1, \dots, b_{d}$ where $a_ib_{j}$ is labeled by $1$ for $i\geq j-1$ and $a_ib_{i+2}$ is labeled by $0$, while all other edges are labeled by $*$. See Figure \ref{fig1} for an illustration.

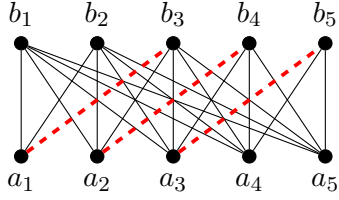
\begin{figure}[ht]
    \begin{center}
	\begin{tikzpicture}
            \node[vertex,label=below:$a_1$] (a1) at (-2,-1) {};
            \node[vertex,label=below:$a_2$] (a2) at (-1,-1) {};
            \node[vertex,label=below:$a_3$] (a3) at (0,-1) {};
            \node[vertex,label=below:$a_4$] (a4) at (1,-1) {};
            \node[vertex,label=below:$a_5$] (a5) at (2,-1) {};

             \node[vertex,label=above:$b_1$] (b1) at (-2,0.5) {};
            \node[vertex,label=above:$b_2$] (b2) at (-1,0.5) {};
            \node[vertex,label=above:$b_3$] (b3) at (0,0.5) {};
            \node[vertex,label=above:$b_4$] (b4) at (1,0.5) {};
            \node[vertex,label=above:$b_5$] (b5) at (2,0.5) {};

            \draw (a5) edge (b5);
            \draw (a5) edge (b4);
            \draw (a5) edge (b3);
            \draw (a5) edge (b2);
            \draw (a5) edge (b1);

            \draw (a4) edge (b5);
            \draw (a4) edge (b4);
            \draw (a4) edge (b3);
            \draw (a4) edge (b2);
            \draw (a4) edge (b1);

            \draw (a3) edge (b4);
            \draw (a3) edge (b3);
            \draw (a3) edge (b2);
            \draw (a3) edge (b1);

            \draw (a2) edge (b3);
            \draw (a2) edge (b2);
            \draw (a2) edge (b1);

            \draw (a1) edge (b2);
            \draw (a1) edge (b1);

            \draw[line width=0.5mm, red, dashed] (a3) -- (b5);
            \draw[line width=0.5mm, red, dashed] (a2) -- (b4);
            \draw[line width=0.5mm, red, dashed] (a1) -- (b3);
        \end{tikzpicture}
    \end{center}
    \caption{The pattern $\Pi_5$, where black edges are labeled 1, and the red dashed edges are labeled 0.}
    \label{fig1}
\end{figure}

We begin by presenting the basic geometric lemma which shows that incidence graphs of points and unit spheres in $\bF^d$ avoid the pattern $\Pi_{d+1}$.

\begin{lemma}\label{lemma:forbidden pattern}
Let $\cP$ be a set of $n$ points and let $\cS$ be a set of $n$ spheres in $\bF^d$. The incidence graph $G(\cP, \cS)$ does not contain the pattern $\Pi_{d+1}$ defined above.
\end{lemma}
\begin{proof}
For contradiction, let us assume that $G(\cP, \cS)$ does contain $\Pi_{d+1}$, whose vertices $a_1, \dots, a_{d+1}$ correspond to points $P_1, \dots, P_{d+1}$ and $b_1, \dots, b_{d+1}$ correspond to unit spheres $S_1, \dots, S_{d+1}$.

We  show by induction that $S_1\cap \dots \cap S_k=S_1\cap U$ for a affine flat $U$ of dimension at most $d-k+1$, for all $k\leq d+1$. For $k=1$ this statement is trivial. For $k\geq 2$, we may assume that $S_1\cap \dots\cap S_{k-1}=S_1\cap U$ for some flat $U$ of dimension at most $d-k+2$. By Lemma~\ref{lemma:intersections of spheres}, there exists a flat $H$ for which $S_1\cap S_k=S_1 \cap H$. Let $V=U \cap H$. Then
$S_1\cap \dots\cap S_{k}=S_1\cap U \cap S_k=S_1\cap (U \cap H)=S_1 \cap V$. The point $P_{k-2}$ is contained in the spheres $S_1, \dots, S_{k-1}$, but not in $S_{k}$. Therefore, $P_{k-2}$ is a point of $U$ not contained in $V$, implying that $\dim(V)\leq \dim(U)-1\leq d-k+1$.

Applying this claim with $k=d+1$, we conclude that there exists a $0$-dimensional affine space $U$ for which $S_1\cap \dots \cap S_{d+1}=S_1\cap U$. In other words, $S_1\cap \dots \cap S_{d+1}$ is a single point. Then, it is not possible for both $P_{d}$ and $P_{d+1}$ to be contained in $S_1\cap \dots \cap S_{d+1}$, presenting a contradiction.
\end{proof}

\begin{proposition}
Let $G$ be a bipartite graph on $2n$ vertices which does not contain $K_{s, s}$ or the pattern $\Pi_{d+1}$. Then, $G$ has at most $O_{d,s}\Big(n^{2-\frac{1}{\lceil d/2\rceil +1}}\Big)$ edges.
\end{proposition}
\begin{proof}
In \cite{MST}, Lemma 3.2, it is proved that a balanced bipartite graph on $2n$ vertices containing no $K_{s, s}$ or $\Pi_{d+1}$ has at most $O_{d,s}\Big(n^{2-\frac{1}{\lceil (d+2)/2\rceil }}\Big)$.
\end{proof}

\subsection{Constructions with many unit distances}

In this section, our goal is to prove Theorem \ref{thm:sphere_construction}. More precisely, for every $n$ and $d$, we construct a set $\cP$ of $n$ points in $\bF_q^d$, where $q=p^r$ is some prime power, with the properties that $\cP$ spans $\Omega(n^{2-\frac{1}{\lceil d/2\rceil+1}})$ unit distances and that the unit distance graph does not contain $K_{s, s}$. 

Our construction is based on so called variety-evasive sets. For fixed parameters $1\leq k\leq d$ and $\Delta, s>0$, we say that a set $U$ is $(k, s)$-variety-evasive if every variety $V\subseteq \bF_q^d$ of dimension $k$ and degree at most $\Delta$ intersects $U$ in less than $s$ points. Here we only consider varieties of degree at most $2$ and therefore we fix $\Delta=2$. Dvir, Koll\'ar and Lovett constructed large variety-evasive sets and in particular the following theorem is a special case of Corollary 6.1 from \cite{DKL}.

\begin{theorem}\label{thm:variety evasive sets}
For every positive integer $d$ and prime power $q$, there exist a constant $s=s(d)$ and set $U$ of $q^{d-k}$ points in $\bF_q^d$ such that any $k$-dimensional variety of degree at most $2$ intersects $U$ in less than $s$ points.
\end{theorem}

Let us now explain the idea behind our constructions. We set $k=\lfloor d/2\rfloor$ and choose $p$ to be a prime with $p\equiv 3\bmod 4$ and $p\approx n^{\frac{1}{\lceil d/2\rceil + 1}}$. Then, we use Theorem~\ref{thm:variety evasive sets} to find a $(k-1, s)$-variety evasive set $U\subseteq \bF_p^d$ of size $p^{d-k+1}=p^{\lceil d/2\rceil + 1}$. Since $U$ is $(k-1, s)$-variety evasive, we are able to show that the unit distance graph does not contain $K_{s, s}$ as a subgraph. However, in order to ensure many unit distances, we need to take the union of $U$ with a random shift, i.e. we set $\cP_x=U\cup (U+x)$ where $x\in \bF_p^d$ is chosen uniformly at random. 

\begin{proposition}\label{prop:many unit distances}
There exists $x\in \bF_p^d$ for which $\cP_x$ spans $\Omega\Big(|\cP_x|^{2-\frac{1}{\lceil d/2\rceil+1}}\Big)$ unit distances with respect to the norm $\|\cdot \|_d$.
\end{proposition}
\begin{proof}
Choose $x$ form the uniform distribution on $\bF_p^d$. The number of unit distances spanned by $\cP_x$ is at least the number of pairs $u, v\in U$ such that $\|u-(v+x)\|_d=1$. For a fixed pair $u, v\in U$, the equation $\|u-v-x\|_d=1$ is quadratic in $x$ and therefore it has $(1+o(1))p^{d-1}$ solutions (see e.g. Theorems 6.26 and 6.27 in \cite{LN}). Hence, for fixed $u, v\in U$ we have $\Pb[\|u-(v+x)\|_d=1]=(1+o(1))p^{-1}$ and so the expected number of pairs $u, v\in U$ with $\|u-(v+x)\|_d=1$ is $(1+o(1))|U|^2/p$. Therefore, $x$ can be chosen such that the number of unit distances spanned by $\cP_x$ is $\Omega(|U|^2/p)=\Omega\Big(|\cP_x|^{2-\frac{1}{\lceil d/2\rceil+1}}\Big)$.
\end{proof}

Let us fix the value $x$ from Proposition~\ref{prop:many unit distances} from now on and write $\cP=\cP_x$. Note that the set $\cP$ is $(k-1, 2s)$-variety evasive, since no variety of dimension $k-1$ and degree at most $2$ can contain more than $s$ points of either $U$ or $U+x$. 

\begin{proposition}
The unit distance graph of $\cP$ with respect to the norm $\|\cdot \|_d$ does not contain $K_{4s, 4s}$ as a subgraph.
\end{proposition}
\begin{proof} 
We argue by contradiction, assuming that there exist points $v_1, \dots, v_{4s},w_1, \dots, w_{4s}\in \cP$ such that $\|v_i-w_j\|_d=1$ for all $i, j\in [4s]$. Let $V=\aff \{v_1, \dots, v_{4s}\}$ and $W=\aff \{w_{1}, \dots, w_{4s}\}$. 

We claim that the flats $V$ and $W$ are orthogonal, with respect to the bilinear form $\langle \cdot, \cdot\rangle_d$. If we denote by $S_1, \dots, S_{4s}$ the unit spheres centered at $w_{1}, \dots, w_{4s}$, the assumption implies that $v_{1}, \dots, v_{4s}\in \bigcap_{i=1}^{4s} S_i$. By Lemma~\ref{lemma:intersections of spheres}, there exists an affine flat $V'$, orthogonal to $W$, such that $\bigcap_{j=1}^{4s} S_j=S_1\cap V'$. Since $V\subseteq V'$, we conclude $V$ and $W$ are orthogonal, which implies that $\dim V+\dim W\leq d$. 

Since $\cP$ is $(k-1, 2s)$-variety evasive, no $(k-1)$-dimensional flat contains $2s$ points of $\cP$ (note that an affine flat is a variety of degree $1$). In particular, this means that $\dim V\geq k$ and $\dim  W\geq k$. Combining this with $\dim V+\dim W\leq d=2k+1$, we conclude that $\dim V = k$ or $\dim W = k$. By symmetry, we may assume that $\dim V=k$. 

We claim that $V\subseteq S_j$, for every $j\in [4s]$. Suppose this was not the case, i.e. that there was some $j$ for which $V\not\subseteq S_j$. Since $V$ is a variety of dimension $k$ and degree $1$ and $S_j$ is a variety of dimension $d-1$ and degree $2$, Theorem~\ref{thm:preliminary} implies that if $V\cap S_j$ is a union of irreducible components $Z_1, \dots, Z_t$, then $\sum_{i=1}^t \deg Z_i\leq 2$. Hence, $V\cap S_j$ consists either of one $(k-1)$-dimensional component of degree $2$ or of at most two $(k-1)$-dimensional components of degree $1$. However, recall that $v_1, \dots, v_{4s}\in S_j\cap V$ and that every $(k-1)$-dimensional variety of degree at most $2$ contains less than $2s$ points of $\cP$. Hence, in both cases we have a contradiction and so we must have $V\subseteq S_j$. Therefore, by Lemma~\ref{lemma:isotropic flats}, $V$ must be a totally isotropic flat and we have $\langle w_j-y, z-y\rangle_d=0$ for all $y, z\in V$. 

Since $\langle w_j-v_i, w_j-v_i\rangle_d =1$, the point $w_j$ does not lie in the flat $V$. Also, $\aff (\{w_j\}\cup V)$ is orthogonal to $V$, since $\langle w_{j}-x, y-x\rangle_d=0$ for all $x, y\in V$. Hence, $\dim \aff (\{w_{j}\}\cup V)=\dim V+1=k+1$. If $d$ is even, this is a contradiction since $\dim V+\dim \aff (\{w_{j}\}\cup V)=2k+1>d$.

If $d$ is odd, we found a totally isotropic $k$-flat with a vector $w_{j}-v_{i}$ which has norm $1$ and is orthogonal to this $k$-flat. But this is impossible by Lemma~\ref{lemma:auxiliary odd d}.
\end{proof}

Let us now put all the ingredients together and show the general construction.

\begin{proof}[Proof of Theorem \ref{thm:sphere_construction}]
Let $p$ be the smallest prime such that $p\equiv3 \bmod 4$ and $p>n^{\frac{1}{\lceil d/2\rceil +1}}$. By a theorem of Breusch \cite{B}, we have $p\leq 2n^{\frac{1}{\lceil d/2\rceil +1}}$, as long as $n^{\frac{1}{\lceil d/2\rceil +1}}\geq 7$. Let $\cP$ be the set of at least $p^{\lceil d/2\rceil +1} \geq n$ points constructed in this section. For $d\not\equiv 1\bmod 4$, taking a random $n$ element subset of $\cP$ gives the desired set. However, for $d\equiv 1 \bmod 4$ the standard norm and the norm $\|\cdot\|_d$ are different as we need one final step to complete the construction.

If $d\equiv 1\bmod 4$, we let $q=p^2$, we choose $\alpha\in \bF_q$ to be a solution of $\alpha^2=-1$, and consider the map $\varphi:\bF_p^d\to \bF_q^d$ given by $\varphi((x_1, \dots, x_d))=(x_1, \dots, x_{d-1}, \alpha x_d)$. Note that $\varphi(u)$ and $\varphi(v)$ are at unit distance measured with respect to the standard inner product if and only if $\|u-v\|_d=1$. Hence, we conclude that the set $\{\varphi(u)|u\in \cP\}\subseteq \bF_q^d$ spans $\Omega\Big(|\cP|^{2-\frac{1}{\lceil d/2\rceil+1}}\Big)$ unit distances and has no $K_{s, s}$ in the unit distance graph, where unit distances are measured with respect to the standard inner product. Now taking a random $n$ element subset of $\cP$ completes the proof.
\end{proof}

\vspace{0.5cm}
\noindent
\textbf{Acknowledgements.} We would like to thank Larry Guth for suggesting to study point-variety incidences and stimulating discussions.

\end{document}